\documentclass[11pt, twoside]{article}
\usepackage{amssymb, amsmath, amsthm}
\usepackage[left=25mm, right=25mm, bottom=30mm]{geometry}
\usepackage{inputenc}
\usepackage{fancyhdr}
\usepackage{tikz}
\usetikzlibrary{positioning}
\usepackage{titling}
\usepackage{hyperref}
\predate{}
\postdate{}
\usepackage{lipsum}
\newtheorem{theorem}{Theorem}
\newtheorem{lemma}[theorem]{Lemma}
\newtheorem{proposition}[theorem]{Proposition}

\newtheorem{claim1}{Claim}

\usepackage{titling}
\usepackage{xcolor}
\usepackage{placeins}
\usepackage{tikz}
\usepackage{verbatim}
\usepackage{titling}
\usepackage[T1]{fontenc}
\usepackage{currvita}
\usepackage{bbm}
\usepackage{enumitem}
\newtheorem{corollary}[theorem]{Corollary}
\newtheorem{claim2}{Claim}
\newtheorem{claim3}{Claim}
\newtheorem{claim4}{Claim}
\newtheorem{claim5}{Claim}

\newtheorem{case1}{Case}

\theoremstyle{definition}

\theoremstyle{remark}

\newcommand{\cF}{\mathcal{F}}
\newcommand{\cG}{\mathcal{G}}

\begin{document}
\newcommand{\Addresses}{
\bigskip
\footnotesize

\medskip

\noindent Maria-Romina~Ivan, \textsc{Department of Pure Mathematics and Mathematical Statistics, Centre for Mathematical Sciences, Wilberforce Road, Cambridge, CB3 0WB, UK,} and\\\textsc{Department of Mathematics, Stanford University, 450 Jane Stanford Way, CA 94304, USA.}\par\noindent\nopagebreak\textit{Email addresses: }\texttt{mri25@dpmms.cam.ac.uk, m.r.ivan@stanford.edu}

\medskip

\noindent Sean~Jaffe, \textsc{Department of Pure Mathematics and Mathematical Statistics, Centre for Mathematical Sciences, Wilberforce Road, Cambridge, CB3 0WB, UK}\par\noindent\nopagebreak\textit{Email address: }\texttt{scj47@cam.ac.uk}}

\pagestyle{fancy}
\fancyhf{}
\fancyhead [LE, RO] {\thepage}
\fancyhead [CE] {MARIA-ROMINA IVAN AND SEAN JAFFE}
\fancyhead [CO] {THE SATURATION NUMBER FOR THE DIAMOND IS LINEAR}
\renewcommand{\headrulewidth}{0pt}
\renewcommand{\l}{\rule{6em}{1pt}\ }
\title{\Large{\textbf{THE SATURATION NUMBER FOR THE DIAMOND IS LINEAR}}}
\author{MARIA-ROMINA IVAN AND SEAN JAFFE}
\date{ }
\maketitle
\begin{abstract}
For a fixed poset $\mathcal P$ we say that a family $\mathcal F\subseteq\mathcal P([n])$ is $\mathcal P$-saturated if it does not contain an induced copy of $\mathcal P$, but whenever we add a new set to $\mathcal F$, we form an induced copy of $\mathcal P$. The size of the smallest such family is denoted by $\text{sat}^*(n, \mathcal P)$.\par For the diamond poset $\mathcal D_2$ (the two-dimensional Boolean lattice), while it is easy to see that the saturation number is at most $n+1$, the best known lower bound has stayed at $O(\sqrt n)$ since the introduction of the area of poset saturation. In this paper we prove that $\text{sat}^*(n, \mathcal D_2)\geq \frac{n+1}{5}$, establishing that the saturation number for the diamond is linear. The proof uses a result about certain pairs of set systems.
\end{abstract}
\section{Introduction}
We say that a poset $(\mathcal Q, \leq')$ contains an
\textit{induced copy} of a poset $(\mathcal P, \leq)$ if there exists an injective function $f:\mathcal P\rightarrow\mathcal Q$ such that $(\mathcal P, \leq)$ and $(f(\mathcal P), \leq')$ are isomorphic. For a fixed poset $\mathcal P$ we call a family $\mathcal F$ of subsets of $[n]$ $\mathcal P$-\textit{saturated} if $\mathcal F$ does not contain an induced copy of $\mathcal P$, but for every subset $S$ of $[n]$ such that $S\notin\mathcal F$, the family $\mathcal F\cup \{S\}$ does contain such a copy. We denote by $\text{sat}^*(n, \mathcal P)$ the size of the smallest $\mathcal P$-saturated family of subsets of $[n]$. In general we refer to $\text{sat}^*(n, \mathcal P)$ as the \textit{induced saturation number} of $\mathcal P$.

Saturation for posets was introduced by Gerbner, Keszegh, Lemons, Palmer, P{\'a}lv{\"o}lgyi and Patk{\'o}s \cite{gerbner2013saturating}, although this was not for \textit{induced} saturation. We refer the reader to the textbook of Gerbner and Patk{\'o}s \cite{gerbner2018extremal} for a nice introduction to the area. Despite its similarity with saturation for graphs, as well as its simple formulation, poset saturation is intrinsically different, and seems very difficult to analyse -- this is partially due to the rigidity of \textit{induced} copies of posets. One of the most important conjectures in the area is that for every poset the saturation number is either bounded or linear \cite{keszegh2021induced}. Indeed, linearity has been proven for just a few special posets. Most notably, the following posets have linear saturation number: $\mathcal V$ (one minimal element below two incomparable) and $\Lambda$ (one maximal element above two incomparable) \cite{ferrara2017saturation}, the butterfly (two maximal elements completely above two minimal elements) \cite{ivan2020saturationbutterflyposet, keszegh2021induced}, and the antichain \cite{bastide2024exact}. 

For the general setting, it was first shown by Keszegh, Lemons, Martin, P\'alv\"olgy and Patk\'os \cite{keszegh2021induced} that the saturation number is either bounded or at least $\log_2(n)$. This was later improved by Freschi, Piga, Sharifzadeh and Treglown \cite{freschi2023induced} who showed that $\text{sat}^*(n,\mathcal P)$ is either bounded or at least $2\sqrt n$. In the other direction, Bastide, Groenland, Ivan and Johnston \cite{polynomial} showed that the saturation number of any poset grows at most polynomially. Moreover, very recently Ivan and Jaffe \cite{gluing} showed that for any poset, one can add at most 3 points to obtain a poset with saturation number at most linear. 

The diamond poset, which we denote by $\mathcal D_2$, is the 4 point poset with one minimal element, one maximal element and two incomparable elements as shown in the picture below. Equivalently, $\mathcal D_2$ is the two-dimensional Boolean lattice.
\begin{center}
\begin{figure}[h]
\centering
\hspace{0.1cm}\\
\begin{tikzpicture}
\node (top) at (4,1) {$\bullet$}; 
\node (left) at (3,0) {$\bullet$};
\node (right) at (5,0) {$\bullet$};
\node (bottom) at (4,-1) {$\bullet$};
\draw (top) -- (left) -- (bottom) -- (right) -- (top);
\end{tikzpicture}
\end{figure}
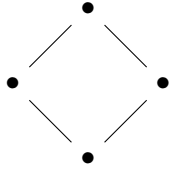
The Hasse diagram of the diamond poset $\mathcal D_2$
\end{center}
Despite the simplicity of the diamond poset, the question of its induced saturation number has eluded all techniques successfully used in poset saturation so far. It is easy to see that $\text{sat}^*(n,\mathcal D_2)\leq n+1$ -- indeed, one can check that a maximal chain is diamond-saturated. For the lower bound, the first non-trivial one was established by Ferrara, Kay, Kramer, Martin, Reiniger, Smith and Sullivan \cite{ferrara2017saturation} who showed that $\text{sat}^*(n,\mathcal D_2)\geq \log_2(n)$. They also conjectured that $\text{sat}^*(n, \mathcal D_2)=\Theta(n)$. Next, Martin, Smith and Walker \cite{martin2020improved} showed that $\text{sat}^*(n,\mathcal D_2)\geq\sqrt n$. We mention that these bounds were not specific to the diamond structure -- they were in fact a consequence of the fact that the diamond is part of a larger class of special posets, namely posets that have the \textit{unique cover twin property} (UCTP). By a more in depth analysis of diamond-saturated families, the most recent lower bound was pushed to $\text{sat}^*(n, \mathcal D_2)\geq(4-o(1))\sqrt n$ by Ivan \cite{ivan2021minimal}.

In this paper we show that $\text{sat}^*(n,\mathcal D_2)\geq\frac{n+1}{5}$, and therefore:
\begin{theorem}
$\text{sat}^*(n,\mathcal D_2)=\Theta(n)$.\end{theorem} Our proof builds on the work of \cite{ivan2021minimal}, which is essentially Section 2. Inspired by the set systems analysed in Section 2, we define and analyse in Section 3 a certain family of pairs of set systems. This is the heart of the paper. Indeed, having completed Section 3, Section 4 is
just a short step to deduce from these results the desired linear lower bound for any diamond-saturated family.

Throughout the paper our notation is standard. For a finite set $X$ we denote by $\mathcal P(X)$ the power set of $X$, and for a positive integer $n$ we denote by $[n]$ the set $\{1,2,\dots,n\}$. Also, $C_2$ is the chain poset of size 2, i.e. two distinct comparable elements. 
\section{Preliminary lemmas}
We start the section with a helpful lemma proved in \cite{ferrara2017saturation}.
\begin{lemma}[Lemma 5 in \cite{ferrara2017saturation}]
Let $\mathcal F\subseteq\mathcal P([n])$ be a diamond-saturated family. If $\emptyset\in\mathcal F$, or $[n]\in\mathcal F$, then $|\mathcal F|\geq n+1$.
\end{lemma}
Therefore, throughout the rest of the section we may assume that $\mathcal F$ is a diamond-saturated family with ground set $[n]$ such that $\emptyset,[n]\notin\mathcal F$.

Let $\mathcal B$ be the set of maximal elements of $\mathcal F$. Let also $\mathcal A_0$ to be the set of elements above a copy of $\mathcal V$ in $\mathcal F$. More precisely, $\mathcal A_0=\{X\in\mathcal P([n]):\exists Y,Z,W\in\mathcal F\text { such that }X,Y,Z,W\text{ form a diamond and }X\\\text{ is the maximal element of that diamond}\}$. Let $\mathcal A_1$ to be the set of minimal elements of $\mathcal A_0$. Finally, let $\mathcal A = \{X \in \mathcal A_1 :B\not\subseteq X, \forall B\in\mathcal B\}$.

Additionally, let $G(\mathcal A)$ be all the sets in $\mathcal F$ that generate $\mathcal A$. More precisely, for every $A\in\mathcal A$, let $G(A)=\{P\in\mathcal F:\exists R,Q\in\mathcal F\text{ such that } A,P,R,Q \text{ form a diamond with } A\text{ as its maximal element}\}$. Then $G(\mathcal A)=\cup_{A\in\mathcal A}G(A)$. Let also $W=\{i\in[n]:i\notin X\text{ for all }X\in G(\mathcal A)\}$. 

\begin{lemma}\label{niceproperty1}
We have that $\mathcal A$ and $\mathcal B$ are disjoint, and $\mathcal A \cup \mathcal B$ is a $C_2$-saturated family of $\mathcal P([n])$.
\end{lemma}
\begin{proof} We will first show that if $A\in \mathcal A$ and $B\in \mathcal B$, then $A\not \subseteq B$. This immediately implies that $\mathcal A$ and $\mathcal B$ are disjoint, since if $Y\in \mathcal A\cap \mathcal B$, then $Y\subseteq Y$, a contradiction. Suppose now that $A\subseteq B$ for some $A\in \mathcal A$ and $B\in \mathcal B$. Then, by definition, there exist elements $E,F,G\in \mathcal F$ that together with $A$ form a diamond, and $A$ is the maximal element of that diamond. This means that $B$, $E$, $F$ and $G$ form an induced copy of $\mathcal D_2$ in $\mathcal F$, a contradiction. This is illustrated below. Thus, $\mathcal A$ and $\mathcal B$ are indeed disjoint.

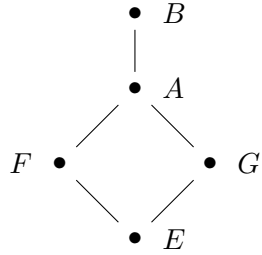
\begin{figure}[h]
\hspace{0.1cm}\\
\centering
\begin{tikzpicture}
\node[label=right:$A$] (top) at (4,1) {$\bullet$}; 
\node[label=left:$F$] (left) at (3,0) {$\bullet$};
\node[label=right:$G$] (right) at (5,0) {$\bullet$};
\node[label=below:$E$] (bottom) at (4,-1) {$\bullet$};
\node [label=right:$B$] (toptwo) at (4,2) {$\bullet$};
\draw (top) -- (left) -- (bottom) -- (right) -- (top);
\draw (top)--(toptwo);
\end{tikzpicture}
\end{figure}
\FloatBarrier
Moreover, since both $\mathcal A$ and $\mathcal B$ are antichains, and, by definition, no element of $\mathcal B$ is contained in any element of $\mathcal A$, we also get that $\mathcal A\cup \mathcal B$ does not contain a copy of $C_2$. We are therefore left to show that whenever we add a new element to $\mathcal A\cup\mathcal B$ we obtain a copy of $C_2$. As an intermediary step we show the following.
\begin{claim1} Let $X\in \mathcal P([n])$. Then either $X\subseteq B$ for some $B\in\mathcal B$, or $A\subseteq X$ for some $A\in\mathcal A_1$.
\end{claim1}
\begin{proof}
If $X\in \mathcal F$, then by the definition of $\mathcal B$, $X$ must be a subset of some element of $\mathcal B$. Thus, we may assume that $X\not\in\mathcal F$. This means that $\mathcal F\cup \{X\}$ contains a copy of a diamond, and $X$ must be an element of that copy. If $X$ is the maximal element of the diamond, then $X\in\mathcal A_0$, so there exists $A\in\mathcal A_1$ such that $A\subseteq X$. This corresponds to the left-most diagram below. If $X$ is not the maximal element of the diamond, then $X\subset H$, where $H$ is the maximal element of the diamond. This case corresponds to the middle and right diagrams below.

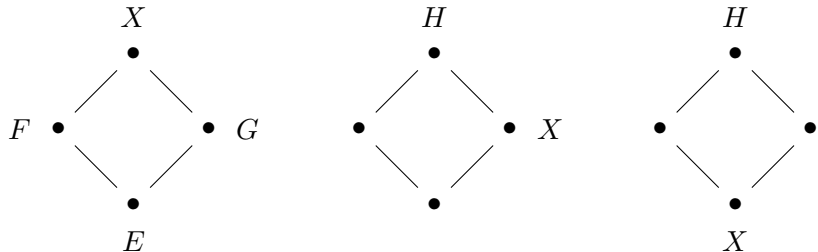
\begin{figure}[h]
\hspace{0.1cm}\\
\centering
\begin{tikzpicture}
\node[label=above:$X$] (top) at (6,1) {$\bullet$}; 
\node[label=left:$F$] (left) at (5,0) {$\bullet$};
\node[label=right:$G$] (right) at (7,0) {$\bullet$};
\node[label=below:$E$] (bottom) at (6,-1) {$\bullet$};
\draw (top) -- (left) -- (bottom) -- (right) -- (top);

\node [label=above:$H$] (top) at (10,1) {$\bullet$}; 
\node (left) at (9,0) {$\bullet$};
\node [label=right:$X$](right) at (11,0) {$\bullet$};
\node (bottom) at (10,-1) {$\bullet$};
\draw (top) -- (left) -- (bottom) -- (right) -- (top);

\node[label=above:$H$] (top) at (14,1) {$\bullet$}; 
\node (left) at (13,0) {$\bullet$};
\node (right) at (15,0) {$\bullet$};
\node[label=below:$X$] (bottom) at (14,-1) {$\bullet$};
\draw (top) -- (left) -- (bottom) -- (right) -- (top);
\end{tikzpicture}
\end{figure}
\FloatBarrier
But in this case we have that $X\subset H$, and $H\subseteq B$ for some $B\in \mathcal B$, thus $X\subset B$, which finishes the proof of the claim.
\end{proof}
Finally, let $X\not\in \mathcal A\cup\mathcal B$. If $X\subset B$ for some $B\in\mathcal B$, then $X$ and $B$ form an induced copy of $C_2$. Hence we may assume that $X\not\subset B$ for all $B\in \mathcal B$. By the previous claim, there exists $A\in\mathcal A_1$ such that $A\subseteq X$. If $A\in \mathcal A$, then $X$ and $A$ form an induced copy of $C_2$ in $\mathcal A\cup\mathcal B\cup\{X\}$. If $A\not \in \mathcal A$, then there exists $C\in\mathcal B$ such that $C \subseteq A$, hence $C$ and $X$ form a copy of $C_2$ in $\mathcal A\cup\mathcal B\cup \{X\}$, which finishes the proof of the lemma.

\end{proof}

\begin{lemma}\label{niceproperty2} We have that $\min\{|B| : B\in \mathcal B\}\geq n-|\mathcal F|$, and $\max\{|A| :A\in \mathcal A\}\leq |\cF|$.
\end{lemma}
\begin{proof}
We begin by showing that $\min \{|B|:B\in \mathcal B\}\geq n-|\mathcal F|$. Let $S\in \mathcal B$. For each element $i\notin S$ we will find an element $X_i\in\mathcal F$ such that $X_i\setminus S=\{i\}$. Since any two such $X_i$ must be different, we get that there are at most $|\mathcal F|$ singletons outside $S$, as desired.

By the maximality of $S$ we have that $S\cup\{i\}\notin\mathcal F$. Therefore, $\mathcal F\cup \{S\cup \{i\}\}$ contains a diamond and $S\cup\{i\}$ must be part of that diamond. Let $A$, $B$ and $C$ be three sets in $\mathcal F$ such that they form a diamond together with $S\cup\{i\}$. By the maximality of $S$, $S\cup \{i\}$ must be the maximal element of the diamond. Let $A$ be the minimal element, and $B$ and $C$ the other two incomparable elements of the diamond.

If $B\neq S$ and $C\neq S $, then one of $B$ and $C$ must contain $i$, otherwise $A$, $B$, $C$ and $S$ induce a diamond in $\mathcal F$, a contradiction. Therefore, without loss of generality, we may assume that $i\in B$. Since $B\subset S\cup\{i\}$, we have $B\setminus S=\{i\}$, as desired. If on the other hand, without loss of generality, $B=S$, since $B$ and $C$ are incomparable, we must have $i\in C$, which again gives us $C\setminus S=\{i\}$.

We will now show that $\max \{|A|:A\in \mathcal A\}\leq |\mathcal F|$. Let $S\in \mathcal A$. For every $i\in S$ we will find an element $Z_i\in\mathcal F$ such that $S\setminus Z_i=\{i\}$. Since by construction these $Z_i$ are pairwise distinct, we can only have at most $|\mathcal F|$ singletons in $S$, finishing the proof.

If $S\setminus \{i\}\in\mathcal F$, we are done by taking $Z_i=S\setminus\{i\}$. Suppose now that $S\setminus\{i\}\notin\mathcal F$. This implies that $\mathcal F\cup\{S\setminus\{i\}\}$ contains an diamond which must have $S\setminus\{i\}$ as an element. 

If $S\setminus \{i\}$ is the maximal element of the diamond, then by definition $S\setminus \{i\}$ is an element of $\mathcal A_0$, contradicting the minimality of $S$. Thus, $S\setminus\{i\}$ is not the maximal element of the diamond. Let $A$ be that maximal element. As $S\in\mathcal A$, $S$ is above a copy of $\mathcal V$ in $\mathcal F$, thus, in order to avoid forming a diamond in $\mathcal F$ together with $A$, we must have $S\not\subseteq A$. However, $S\setminus\{i\}\subset A$, which implies that $i\notin A$, and consequently $S\setminus A=\{i\}$, as desired.
\end{proof}

\begin{lemma}\label{niceproperty3}
If $|\mathcal F|<\frac{n}{2}$, then for at least $n +1- |\mathcal F|$ values of $i\in [n]$, there exists some $B\in \mathcal B$ for which $i\notin B$. Moreover, for at least $n+1 - 2|\mathcal F|$ values of $i\in [n]$, there exist both $A\in \mathcal A$ such that $i\in A$, and  $B\in\mathcal B$ such that $i\notin B$.
\end{lemma}
The key to proving Lemma~\ref{niceproperty3} is the following. 
\begin{proposition} For each $i\in [n]$, either there exists some $W\in \mathcal P([n])$ such that $i\in W$ and $W\setminus\{i\},W\in \mathcal F$, or $i\notin X$ for some $X\in \mathcal B$.
\label{nicelemma4}
\end{proposition}
\begin{proof}
Let $i\in[n]$, and $S\in\mathcal B$ be an element of maximum size. If $i\notin S$ we are done by setting $X=S$, so we may assume that $i\in S$. If $S\setminus\{i\}\in\cF$, we are also done by setting $W = S$. Therefore we may assume that $S\setminus\{i\}\notin\mathcal F$. This means that $\mathcal F\cup\{S\setminus\{i\}\}$ contains an induced copy of a diamond, which must contain $S\setminus \{i\}$. Let $A,B,C$ be elements in $\mathcal F$ such that $S\setminus\{i\},A,B,C$ forms a diamond. If $S\setminus \{i\}$ is the minimal element of the diamond and $A$ is the maximal element of the diamond, then $S\setminus \{i\}\subsetneq B\subsetneq A$. Thus $|S\setminus \{i\}|\leq |A|-2$, which implies that $|S|\leq |A|-1$, contradicting  the fact that $S$ has maximal size. 

If $S\setminus \{i\}$ is the maximal element of the diamond, then $S, A, B$ and $C$ would form a diamond in $\mathcal F$, as illustrated below, a contradiction.
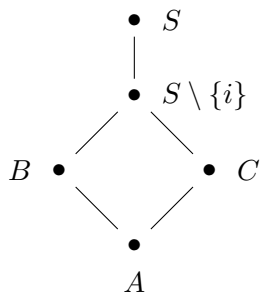
\begin{figure}[hbt!]
\hspace{0.1cm}\\
\centering
\begin{tikzpicture}
\node[label=right:$S\setminus\{i\}$] (top) at (4,1) {$\bullet$}; 
\node[label=left:$B$] (left) at (3,0) {$\bullet$};
\node[label=right:$C$] (right) at (5,0) {$\bullet$};
\node[label=below:$A$] (bottom) at (4,-1) {$\bullet$};
\node [label=right:$S$] (toptwo) at (4,2) {$\bullet$};
\draw (top) -- (left) -- (bottom) -- (right) -- (top);
\draw (top)--(toptwo);
\end{tikzpicture}
\end{figure}
  \FloatBarrier
Therefore $S\setminus\{i\}$ must be one of the two elements in the middle layer of the diamond. As in the diagram below, let $A$ be the unique maximal element and $C$ the unique minimal element of such a diamond.
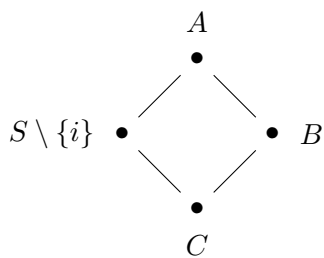
\begin{figure}[hbt!]
\hspace{0.1cm}\\
\centering
\begin{tikzpicture}
\node[label=above:$A$] (top) at (4,1) {$\bullet$}; 
\node[label=left:$S\setminus\{i\}$] (left) at (3,0) {$\bullet$};
\node[label=right:$B$] (right) at (5,0) {$\bullet$};
\node[label=below:$C$] (bottom) at (4,-1) {$\bullet$};
\draw (top) -- (left) -- (bottom) -- (right) -- (top);

\end{tikzpicture}
\end{figure}
\FloatBarrier
We now split our analysis into two cases based on whether or not $A=S$. 
\begin{case1}
$A\neq S$.
\end{case1}
Since $S\setminus \{i\}\subset A$, we have that $|S|\leq |A|$. Since $S$ was chosen to be a set of $\mathcal F$ of maximal size, we must have that $A\in \mathcal B$. Moreover, since $A\neq S$, and they have the same size, we have that $A$ and $S$ must be incomparable. As $S\setminus \{i\}\subset A$, we get that $i\notin A$, as desired.
\begin{case1}
$A= S$. 
\end{case1}
Without loss of generality, let $B$ be of minimal cardinality with respect to the above configuration. In other words, $B$ is a set of minimal cardinality such that $B\in\mathcal F$, and $B$, $S$, $S\setminus\{i\}$ and $C$ form a diamond for some $C\in \mathcal F$. 

Since $B\subset S$ and $S\setminus\{i\}$ and $B$ are incomparable, we must have $i\in B$. If $B\setminus \{i\}\in\cF$ we are done by setting $W=B$. Thus, we may assume that $B\setminus\{i\}\notin\mathcal F$. We therefore have that $\mathcal F\cup\{B\setminus \{i\}\}$ contains a diamond which uses $B\setminus\{i\}$ as an element. Let $P,Q,R$ be elements in $\mathcal F$ such that $B\setminus\{i\}$, $P$, $Q$ and $R$ form a diamond. Note that since $C\subset S\setminus\{i\}$, $i\notin C$, so $C\subseteq B\setminus\{i\}$. We now observe that $B\setminus \{i\}$ cannot be the minimal element of the diamond as otherwise $C$, $P$, $Q$ and $R$ form a diamond in $\mathcal F$. Similarly, $B\setminus\{i\}$ cannot be the maximal element of the diamond as otherwise $S$, $P$, $Q$ and $R$ would form a diamond in $\mathcal F$. Therefore, as illustrated below, $B\setminus\{i\}$ must be one of the elements in the middle layer of the diamond.
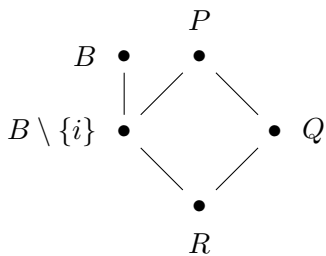
\begin{figure}[hbt!]
\hspace{0.1cm}\\
\centering
\begin{tikzpicture}
\node[label=above:$P$] (top) at (4,1) {$\bullet$}; 
\node[label=left:$B\setminus\{i\}$] (left) at (3,0) {$\bullet$};
\node[label=right:$Q$] (right) at (5,0) {$\bullet$};
\node[label=below:$R$] (bottom) at (4,-1) {$\bullet$};
\node[label=left:$B$] (extra) at (3,1) {$\bullet$};
\draw (top) -- (left) -- (bottom) -- (right) -- (top);
\draw (extra)--(left);
\end{tikzpicture}
\end{figure}
\FloatBarrier
Without loss of generality, we may assume that $P$ is a maximal element of $\mathcal F$, i.e. $P\in \mathcal B$. 
\begin{claim2}
$B$ and $P$ are incomparable.
\end{claim2}
\begin{proof}
Suppose that $P\subseteq B$. Since $B\subsetneq S$, this implies that $P\subsetneq S$, contradicting the maximality of $P$.

Suppose now that $B\subseteq P$. If $B$ and $Q$ were incomparable, then $R,B,Q,P$ would form a diamond in $\mathcal F$. Therefore we either have $B\subseteq Q$, or $Q\subseteq B$. Since $B\setminus\{i\}$ and $Q$ are incomparable, we cannot have $B\subseteq Q$. Thus $Q\subsetneq B$, which implies that $i\in Q$.

We now notice that $Q\subseteq B\subset S$, and since $i\in Q$ we also have that $Q$ and $S\setminus\{i\}$ are incomparable. Moreover, $R\subset B\setminus\{i\}\subset S\setminus\{i\}$. This means that we have the diamond below, which contradicts the minimality of $B$, finishing the proof of the claim.
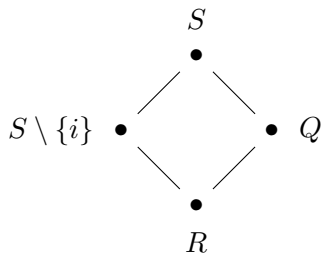
\begin{figure}[hbt!]
\hspace{0.1cm}\\
\centering
\begin{tikzpicture}
\node[label=above:$S$] (top) at (4,1) {$\bullet$}; 
\node[label=left:$S\setminus\{i\}$] (left) at (3,0) {$\bullet$};
\node[label=right:$Q$] (right) at (5,0) {$\bullet$};
\node[label=below:$R$] (bottom) at (4,-1) {$\bullet$};
\draw (top) -- (left) -- (bottom) -- (right) -- (top);
\end{tikzpicture}
\end{figure}
\FloatBarrier
\end{proof}
Since $B\setminus\{i\}\subset P$, and $B$ and $P$ are incomparable, we must have that $i\notin P$, completing the proof in the case where $A=S$.
\end{proof}

We are now ready to prove Lemma \ref{niceproperty3}.
\begin{proof}[Proof of Lemma \ref{niceproperty3}]
Let $M=\{i\in [n] : \text{there exists some $W\in \mathcal P([n])$}$ such that $i\in W$ and $W,W\setminus \{i\}\in \mathcal F\}$, and $N= \cup_{B\in \mathcal B}([n]\setminus B)$. By Proposition~\ref{nicelemma4} we have that $M\cup N = [n]$, hence $|N|\geq n - |M|$.

\begin{claim3} $|M|\leq |\mathcal F|-1$.
\end{claim3}
\begin{proof}
For each $i\in M$, we pick a representative $W_i\in \mathcal F$ such that $i\in W_i$ and $W_i, W_i\setminus\{i\}\in \mathcal F$. Let $G$ be the graph with vertex set $\mathcal F$ such that $\{X,Y\}\in E(G)$ if and only $\{X,Y\}=\{W_i, W_i\setminus\{i\}\}$ for some $i\in M$. By construction $G$ has exactly $|M|$ edges as $\{W_i,W_i\setminus\{i\}\} \neq \{W_j,W_j\setminus\{j\}\}$ for all $i \neq j$. We will show that $G$ is acyclic. 

Suppose that $G$ contains a cycle $A_1,A_2,\dots,A_k$ for some $k\geq 3$. Then there exist distinct singletons $i_1,\dots,i_k$ such that $A_{j+1} = A_j \pm \{i_j\}$ for $j\in [k-1]$, and $A_1 = A_k \pm \{i_k\}$. This means that for all $j\neq k$, $i_k\in A_j$ if and only if $i_k\in A_{j+1}$. Hence, by transitivity, $i_k\in A_1$ if and only if $i_k\in A_k$. But this is a contradiction as $A_1 = A_k \pm \{i_k\}$. 

Therefore, $G$ is an acyclic graph, which means that $|E(G)|\leq |V(G)|-1$. Thus $|M|\leq |\mathcal F|-1$.
\end{proof}
We therefore have $|N|\geq n - |M|\geq n + 1 - |\mathcal F|$, which means that there are at least $n+1-|\mathcal F|$ singletons $i\in[n]$ for which there exists $B\in\mathcal B$ such that $i\notin B$, proving the first part of Lemma~\ref{niceproperty3}.

We now move on to the second part of the Lemma~\ref{niceproperty3}

For every $i\in N$ let $B_i\in \mathcal B$ be an element of minimum size of $\{X\in \mathcal B: i\notin X\}$. Let $\Phi = \{i \in N : B_i = B_j \text{ for some $j\in N\setminus\{i\}$}\}$. By construction, we have $|\Phi|\geq |N| - |\mathcal B|\geq n+1-2|\mathcal F|$. We will show that if $i\in \Phi$, then $i\in A$ for some $A\in\mathcal A$, which will finish the proof.

Towards a contradiction, suppose there exists $i\in \Phi$ such that $i\notin A$ for all $A\in\mathcal A$. Since $\mathcal A\cup\mathcal B$ is $C_2$-saturated, for every $j\in B_i$, there exists some $X_j\in\mathcal A\cup\mathcal B$ such that $X_j $ is comparable (or equal) to $(B_i\setminus\{j\})\cup \{i\}$. This is because if $(B_i\setminus\{j\})\cup \{i \}\in \mathcal A\cup\mathcal B$, then we may set $X_j = (B_i\setminus\{j\})\cup\{i \}$. Otherwise, let $X_j$ be an element of $\mathcal A\cup\mathcal B$ that induces a copy of $C_2$ with $(B_i\setminus\{j\})\cup \{i\}$.

If $X_j \subsetneq (B_i\setminus\{j\})\cup \{i\}\), then $|X_j|<|B_i|$. By assumption, $B_i = B_{i_2}$ for some $i_2\in N \setminus \{i\}$. Therefore, $i_2\notin X_j$ and $|X_j|<|B_{i_2}|$, which, by the minimality of $|B_{i_2}|$, implies that $X_j\in \mathcal A$, and so, by assumption, $i\notin X_j$. However, this would mean that $X_j\subset B_i$, which is a contradiction since $X_j$ and $B_i$ would induce a $C_2$ in $\mathcal A\cup\mathcal B$. Therefore, $(B_i\setminus\{j\})\cup\{i\}\subseteq X_j$ for all $j\in B_i$.

We observe that $X_j\neq B_i$ which means that $X_j$ and $B_i$ are incomparable. Since $B_i\setminus\{j\}\subseteq X_j$, we must have $j\notin X_j$. This means that for all $j\in B_i$ the elements $X_j$ are pairwise distinct. Furthermore, by Lemma~\ref{niceproperty2}, we have that $n-|\mathcal F| \leq|B_i|\leq|X_j|$. Since $|\mathcal F|<\frac{n}{2}$, we must have that $X_j\in\mathcal B$. But then $\{X_j : j\in B_i\}$ is a set of $|B_i| >\frac{n}{2}$ elements of $\mathcal B$. Since $\mathcal B\subseteq \mathcal F$, this contradicts the assumption that $|\mathcal F|<\frac{n}{2}$, completing the proof of the lemma.
\end{proof}
  
\section{Lower bounding a general class of set systems}
In this section we analyse a general class of pairs of set systems. The main result of this section is Lemma \ref{boundingsizecalAcalB}. We then later apply this result to our diamond-saturated family in order to lower bound $|G(\mathcal A)\cup \mathcal B|$. We begin with some definitions.

Let $X$ be a set of size $n$, and $m$ a positive integer such that $2m+1\leq |X|=n$. We define $\mathcal L(X,m)$ to be the set of disjoint pairs of subsets of $\mathcal P(X)$, $(\mathcal G,\mathcal H)$ such that all elements of $\mathcal G$ have size at most $m$, all elements of $\mathcal H$ have size at least $n-m$, and for every $i\in X$, there exists an element $G\in \mathcal G$ such that $i\in G$. Additionally, any induced copy of $C_2$ in $\mathcal G\cup \mathcal H$ must be contained in $\mathcal H$, and for any $A\in \mathcal P(X)$ with $m\leq|A|\leq|X|-m$, there exists an element $B\in\mathcal G\cup\mathcal H$ such that either $B\subseteq A$, or $A\subseteq B$. 

We can think of $\mathcal L(X,m)$ as being pairs of low-level sets that cover the ground set, the $\mathcal G$, that are incomparable to high-level elements, the $\mathcal H$, with the additional property that adding any new element to $\mathcal G\cup\mathcal H$ that is (essentially) between $\mathcal G$ and $\mathcal H$ creates a copy of a $C_2$.

Next, let $\mathcal I,\mathcal J\subseteq \mathcal P(X)$ and define $\mathcal V_0(\mathcal I,\mathcal J)=\{A\in \mathcal P(X):B\not\subseteq A\text{  for all }B\in\mathcal J, \text{ and }\exists P, Q, R\in\mathcal I \text{ such that } A,P,Q,R \text{ form a diamond, with } A \text { the maximal element}\}$.

We define $\mathcal V(\mathcal I,\mathcal J)$ to be the set of minimal elements of $\mathcal V_0(\mathcal I, \mathcal J)$.

Finally, define $\mathcal L^*(X,m)$ to be the set of pairs of $\mathcal P(X)$, $(\mathcal I,\mathcal J)$, such that $(\mathcal V(\mathcal I, \mathcal J), \mathcal J)\in\mathcal L(X,m)$, and all sets of $\mathcal I$ have size at most $m$. Lastly, define $f(n,m) = \min \{|\mathcal I\cup\mathcal J|:(\mathcal I,\mathcal J)\in\mathcal L^*([n],m)\}$.   

We are now ready to prove the main result of this section.
\begin{lemma}\label{boundingsizecalAcalB} For any positive integers $n$ and $m$ such that $n\geq 2m+1$, we have \(f(n,m)\geq n - 2m\).
\end{lemma}
\begin{proof}
Our proof will be by induction on $n$. When $n=1$ the result is trivially true. Assume now that $n>1$ and that the claim is true for all $n_0<n$. In other words, for all $n_0<n$ and all $m\in\mathbb N$ such that $n_0\geq2m+1$, we have that $f(n_0,m)\geq n_0-2m$.

Let $(\mathcal I,\mathcal J)\in \mathcal L^*([n],m)$ and assume that $|\mathcal I|+|\mathcal J|<n-2m$. We recall that this means that all the sets in $\mathcal I$ have size at most $m$ and $(\mathcal V(\mathcal I, \mathcal J),\mathcal J)\in\mathcal L([n],m)$. This immediately implies that all sets of $\mathcal J$ have size at least $n-m$, thus $\mathcal I$ and $\mathcal J$ are disjoint.

For every $i\in [n]$ we define $\mathcal G_i$ to be the family $\{A\in \mathcal V(\mathcal I, \mathcal J): i\in A\}$. Let $j\in[n]$ be such that $\mathcal G_j$ is a minimal element of the poset $\{\mathcal G_1, \dots,\mathcal G_n\}$ (ordered by inclusion). By reordering the ground set if necessary, we may assume that the set $\{k\in [n] : \mathcal G_k = \mathcal G_j\}=[t]$, for some $t\in [n]$. We will use this $t$ to pass to the same set-up, but for a smaller ground set. In order to do so, we define $\widehat{\mathcal I} = \{I\in \mathcal I :[t]\cap I= \emptyset\}$, $\widehat{\mathcal J}= \{J\setminus [t] : J\in\mathcal J, [t-1]\subseteq J\}$, and $\widehat{\mathcal V(\mathcal I, \mathcal J)} = \{A\in\mathcal V(\mathcal I, \mathcal J) : [t]\cap A = \emptyset\}$. We first show that this restriction preserves the structure.
\begin{claim4}
\label{claim1}
If $n\geq 2m+t+1$, then $\mathcal V(\widehat{\mathcal I},\widehat{\mathcal J})=\widehat{\mathcal V(\mathcal I, \mathcal J)}$.
\end{claim4}
\begin{proof}
For one direction, let $A\in\widehat{\mathcal V(\mathcal I, \mathcal J)}$. We will prove that $A\in\mathcal V(\widehat{\mathcal I},\widehat{\mathcal J})$. Since $A\in\mathcal V(\mathcal I,\mathcal J)$, there exist $P,Q,R\in\mathcal I$ which satisfy the following diagram:

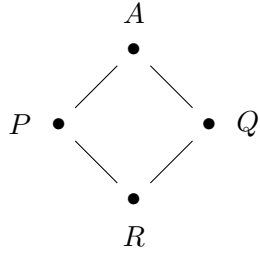
\begin{figure}[h]
\hspace{0.1cm}\\
\centering
\begin{tikzpicture}
\node[label=above:$A$] (top) at (4,1) {$\bullet$}; 
\node[label=left:$P$] (left) at (3,0) {$\bullet$};
\node[label=right:$Q$] (right) at (5,0) {$\bullet$};
\node[label=below:$R$] (bottom) at (4,-1) {$\bullet$};
\draw (top) -- (left) -- (bottom) -- (right) -- (top);
\end{tikzpicture}
\end{figure}
\FloatBarrier
Moreover, since $[t]\cap A=\emptyset$, we also have that $P$, $Q$ and $R$ have empty intersection with $[t]$, thus by definition, $P,Q,R\in\widehat{\mathcal I}$. Hence $A$ is the maximal element of a diamond where the other 3 elements are in $\widehat{\mathcal I}$. 

Suppose for a contradiction that there exists $B\subseteq A$ for some $B\in\widehat{\mathcal J}$. Then $|B|\geq |B^*|-t$ for some $B^*\in\mathcal J$. However, since $A\in\mathcal V(\mathcal I, \mathcal J)$ and $(\mathcal V(\mathcal I, \mathcal J),\mathcal J)\in\mathcal L([n],m)$, we have that $|A|\leq m $ and $|B^*|\geq n-m$. But now we have $m\geq |A|\geq |B|\geq|B^*|-t\geq n-m-t$, which is a contradiction. Thus, no element of $\widehat{\mathcal J}$ is a subset of $A$.

Finally, suppose there exists $Z\subsetneq A$ such that $Z\in\mathcal V_0(\widehat{\mathcal I},\widehat{\mathcal J} )$. Therefore, there exist $ W,X,Y\in \widehat{\mathcal I}$ such that $Z, W, X, Y$ form a diamond where $Z$ is the top element, as illustrated below.
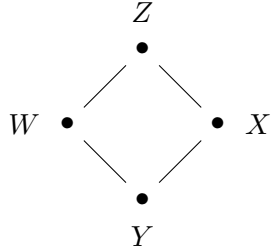
\begin{figure}[h]
\hspace{0.1cm}\\
\centering
\begin{tikzpicture}
\node[label=above:$Z$] (top) at (4,1) {$\bullet$}; 
\node[label=left:$W$] (left) at (3,0) {$\bullet$};
\node[label=right:$X$] (right) at (5,0) {$\bullet$};
\node[label=below:$Y$] (bottom) at (4,-1) {$\bullet$};
\draw (top) -- (left) -- (bottom) -- (right) -- (top);
\end{tikzpicture}
\end{figure}
\FloatBarrier
As argued before, since $Z$ is a subset of $A$, this will imply that $W,X,Y$ are elements of $\mathcal I$. Moreover, $Z$ cannot contain an element of $\mathcal J$ since $A$ does not contain one. Consequently this would mean that $A\notin\mathcal V(\mathcal I,\mathcal J)$ as $Z$ contradicts its minimality. Thus no such $Z$ exists, proving that $\widehat{\mathcal V(\mathcal I, \mathcal J)}\subseteq \mathcal V(\widehat{\mathcal I},\widehat{\mathcal J})$. 

For the other direction, let $A\in \mathcal V(\widehat{\mathcal I},\widehat{\mathcal J})$. Then there exist $P,Q,R\in\widehat{\mathcal I}$ such that $A,P,Q,R$ form a diamond in which $A$ is the top point, as illustrated below.
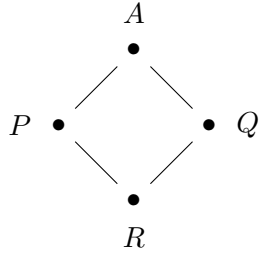
\begin{figure}[h]
\hspace{0.1cm}\\
\centering
\begin{tikzpicture}
\node[label=above:$A$] (top) at (4,1) {$\bullet$}; 
\node[label=left:$P$] (left) at (3,0) {$\bullet$};
\node[label=right:$Q$] (right) at (5,0) {$\bullet$};
\node[label=below:$R$] (bottom) at (4,-1) {$\bullet$};
\draw (top) -- (left) -- (bottom) -- (right) -- (top);
\end{tikzpicture}
\end{figure}
\FloatBarrier
Consider $P\cup Q$. If some $B\in\widehat{\mathcal J}$ would be a subset of $P\cup Q$, then it would also be a subset of $A$, a contradiction. Thus, by minimality, $A=P\cup Q$. Since $[t]\cap P = [t]\cap Q=\emptyset$, we also have $[t]\cap A = \emptyset$. Since $P,Q,R\in\widehat{\mathcal I}\subseteq\mathcal I$, it is enough to show that no set of $\mathcal J$ is contained in $A$, and that $A$ is minimal with respect to the diamond configuration.

Suppose that $B\subseteq A$ for some $B\in\mathcal J$. Since $A\cap [t]=\emptyset$, we have that $B\in\widehat{\mathcal J}$, which is a contradiction.

Finally, suppose that there exist $W,X,Y\in\mathcal I$ and $Z\subsetneq A$ such that $W,X,Y, Z$ form a diamond with $Z$ being its top element, as depicted below.
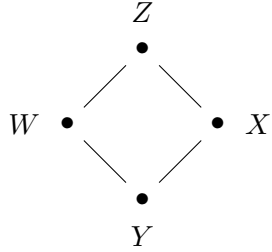
\begin{figure}[h]
\hspace{0.1cm}\\
\centering
\begin{tikzpicture}
\node[label=above:$Z$] (top) at (4,1) {$\bullet$}; 
\node[label=left:$W$] (left) at (3,0) {$\bullet$};
\node[label=right:$X$] (right) at (5,0) {$\bullet$};
\node[label=below:$Y$] (bottom) at (4,-1) {$\bullet$};
\draw (top) -- (left) -- (bottom) -- (right) -- (top);
\end{tikzpicture}
\end{figure}
\FloatBarrier
Since $A\cap[t]=\emptyset$, we also have that $W\cap[t]=X\cap[t]=Y\cap[t]=\emptyset$, thus $W,X,Y\in\widehat{\mathcal I}$. Moreover, $Z$ does not contain an element of $\widehat{\mathcal J}$ as $A$ does not contain one, contradicting the minimality of $A$ in $\mathcal V(\widehat{\mathcal I}, \widehat{\mathcal J})$. We therefore have that $\mathcal V(\widehat{\mathcal I},\widehat{\mathcal J})\subseteq\widehat{\mathcal V(\mathcal I, \mathcal J)}$, which finishes the claim.
\end{proof}
The next property we would like to preserve is the following.
\begin{claim4} If $n \geq 2m+t+1$, then $(\widehat{\mathcal V(\mathcal I,\mathcal J)},\widehat{\mathcal J})\in \mathcal L([n]\setminus [t] ,m)$.
\label{claim2}
\end{claim4}
\begin{proof}
To begin with, it is clear that the ground set is $[n]\setminus[t]$. Since $(\mathcal V(\mathcal I, \mathcal J),\mathcal J)\in\mathcal L([n],m)$, we have that all elements of $\mathcal V(\mathcal I, \mathcal J)$, and so all the elements of $\widehat{\mathcal V(\mathcal I, \mathcal J)}$, have size at most $m$. Also, by construction, all elements of $\widehat{\mathcal J}$ have size at least $n-m-t=|[n]\setminus[t]|-m$. Since $m<n-m-t$, this also implies that $\widehat{\mathcal V(\mathcal I, \mathcal J)}$ and $\widehat{\mathcal J}$ are disjoint. 

Next, suppose that we have two sets $E,F\in\widehat{\mathcal V(\mathcal I, \mathcal J)}\cup\widehat{\mathcal J}$ such that $E\subsetneq F$. Since $\mathcal V(\mathcal I, \mathcal J)$ is $C_2$-free, then so is $\widehat{\mathcal V(\mathcal I, \mathcal J)}$, hence at least one of $E$ or $F$ must be in $\widehat{\mathcal J}$. Suppose that one of them is not in $\widehat{\mathcal J}$. As every set in $\widehat{\mathcal V(\mathcal I, \mathcal J)}$ has size less than the size of any set in $\widehat{\mathcal J}$, we must have $E\in\widehat{\mathcal V(\mathcal I, \mathcal J)}\subseteq \mathcal V(\mathcal I, \mathcal J)$, and $F\in\widehat{\mathcal J}$. By definition, there exists an element $F^*\in\mathcal J$ such that $F = F^*\setminus[t]$. Since $E\subsetneq F\subseteq F^*$, we get that $E$ and $F^*$ form an induced copy of $C_2$ in $\mathcal V(\mathcal I, \mathcal J)\cup\mathcal J$ that is not completely contained in $\mathcal J$, a contradiction.

Now, let $E\in\mathcal P([n]\setminus[t])$ such that $m\leq|E|\leq n -m-t$. Consider the set $E\cup[t-1]$ which has size at least $m$ and less than $n-m$. Since $(\mathcal V(\mathcal I, \mathcal J),\mathcal J)\in \mathcal L([n],m)$, there exists $G\in \mathcal V(\mathcal I, \mathcal J)$ such that $G\subseteq E\cup[t-1]$, or there exists $H\in \mathcal J$ such that $E\cup[t-1]\subseteq H$. If $G\subseteq E\cup[t-1]$ for some $G\in \mathcal V(\mathcal I, \mathcal J)$, then $t\not\in G$. This means that $G\notin\mathcal G_t$. Since $\mathcal G_t=\mathcal G_k$ for all $k\in [t-1]$, we get that $k\notin G$ for all $k\in[t]$. This means that $G\cap [t]=\emptyset$, thus $G\in\widehat{\mathcal V(\mathcal I, \mathcal J)}$. If on the other hand $E\cup [t-1]\subseteq H$ for some $H\in\mathcal J$, then $[t-1]\subseteq H$, and $E\subseteq H\setminus[t]$ (as $t\notin E)$. Since $H\setminus [t]\in \widehat{\mathcal J}$, we conclude that $E$ is either a subset of $\widehat{\mathcal V(\mathcal I, \mathcal J)}$, or a subset of $\widehat{\mathcal J}$.

We are left to show that $\widehat{\mathcal V(\mathcal I, \mathcal J)}$ is a cover of $[n]\setminus[t]$. Let $x\in[n]\backslash [t]$. By construction, we have that $\mathcal G_x\neq \mathcal G_1$, and that $\mathcal G_x\not\subset \mathcal G_1$. Thus, there exists an element $M\in\mathcal G_x\setminus\mathcal G_1$. In particular, $x\in M$. Moreover, since $\mathcal G_1=\mathcal G_k$ for all $k\in [t]$, we have that $M\cap [t]=\emptyset$, thus by definition $M\in\widehat{\mathcal V(\mathcal I, \mathcal J)}$, and $x\in M$. Since $x$ was arbitrary, this shows that $\widehat{\mathcal V(\mathcal I, \mathcal J)}$ is a cover of $[n]\setminus[t]$, completing the proof of the claim.
\end{proof}
\begin{claim4}If $n \geq 2m+t+1$, then $(\widehat{\mathcal I},\widehat{\mathcal J})\in \mathcal L^*([n]\setminus [t],m)$.
\end{claim4}
\begin{proof}
In order to prove the claim we have to show that all elements of $\widehat{\mathcal I}$ have size at most $m$, and that $(\mathcal V(\widehat{\mathcal I}, \widehat{\mathcal J}),\widehat{\mathcal J})\in\mathcal L([n]\setminus [t],m)$. The first condition is trivially true as $\widehat{\mathcal I}$ is a subset of $\mathcal I$, and all sets of $\mathcal I$ have size at most $m$.

The second condition is also true as by Claim~\ref{claim1} we get that $\mathcal V(\widehat{\mathcal I},\widehat{\mathcal J})=\widehat{\mathcal V(\mathcal I, \mathcal J)}$, which together with Claim~\ref{claim2} gives $(\mathcal V(\widehat{\mathcal I},\widehat{\mathcal J}),\widehat{\mathcal J})=(\widehat{\mathcal V(\mathcal I,\mathcal J)},\widehat{\mathcal J})\in \mathcal L([n]\setminus [t] ,m)$, which finished the proof.
\end{proof}
\begin{claim4}
$|\mathcal I|\geq |\widehat{\mathcal I}|+1$.
\end{claim4}
\begin{proof}
Since $((\mathcal V(\mathcal I,\mathcal J),\mathcal J)\in \mathcal L([n],m)$, there exists an element $G\in\mathcal V(\mathcal I,\mathcal J)$ such that $1\in G$. Moreover, by definition, there exists elements $X,Y,Z\in \mathcal I$ such that $G, X, Y, Z$ form a diamond with $G$ being its top element, as illustrated below.
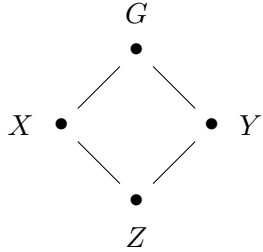
\begin{figure}[h]
\hspace{0.1cm}\\
\centering
\begin{tikzpicture}
\node[label=above:$G$] (top) at (4,1) {$\bullet$}; 
\node[label=left:$X$] (left) at (3,0) {$\bullet$};
\node[label=right:$Y$] (right) at (5,0) {$\bullet$};
\node[label=below:$Z$] (bottom) at (4,-1) {$\bullet$};
\draw (top) -- (left) -- (bottom) -- (right) -- (top);
\end{tikzpicture}
\end{figure}
\FloatBarrier
As argued previously, by the minimality of $G$, we must have $G=X\cup Y$. Therefore, at least one of $X$ or $Y$ must contain 1. Assume without loss of generality that $1\in X$. We then have that $X$ is an element of $\mathcal I$, but it does not have empty intersection with $[t]$, hence it is not an element of $\widehat{\mathcal I}$. Thus $\widehat{\mathcal I}\subsetneq\mathcal I$, which finishes the proof.

\end{proof}
\begin{claim4}$|\mathcal J|\geq |\widehat{\mathcal J}|+t-1$.
\end{claim4}
\begin{proof}
If $t=1$, the claim is trivially true. Thus, we may assume that $t>1$. We clearly have that $|\{H\in\mathcal J : [t-1]\subseteq H\}|\geq |\widehat{\mathcal J}|$, which implies that $|\mathcal J|\geq |\widehat{\mathcal J}| + |\{H\in\mathcal J : [t-1]\not\subseteq H\}|$. We will show that for every $i\in [t-1]$, there exists an element $H\in \mathcal J$ such that $[t-1]\cap H = [t-1]\setminus\{i\}$. This gives $t-1$ distinct elements in $\{H\in\mathcal J : [t-1]\not\subseteq H\}$, which completes the proof.

Let $i\in[t-1]$ be an arbitrary element. Suppose that no element $H\in\mathcal J$ has $H\cap [t-1] = [t-1]\setminus\{i\}$. Since $(\mathcal V(\mathcal I,\mathcal J),\mathcal J)\in\mathcal L([n],m)$, there exists $A\in\mathcal V(\mathcal I,\mathcal J)$ such that $i\in A$. This implies that $k\in A$ for all $k\in [t]$ as $\mathcal G_k=\mathcal G_i$ for all $i,k\in[t]$. Furthermore, we also have that $A$ has size at most $m$.
    
We now define $S_A=\{A\setminus\{i\}\cup B : B\in \binom{[n]\setminus A}{m-|A|+1}\}$. Let $D\in S_A$. By construction, $D$ has size $m$, which means that there exists an element $X\in\mathcal V(\mathcal I,\mathcal J)\cup\mathcal J$ such that either $X\subseteq D$, or $D\subseteq X$. Suppose first that $D\subseteq X$. By construction, $D$, and consequently $X$, must contain at least one element of $[n]\setminus A$, which implies that $X\neq A$. If $i\in X$, since $A\setminus\{i\}\subseteq D\subseteq X$, then $A$ and $X$ will form a copy of $C_2$ in $\mathcal V(\mathcal I,\mathcal J)\cup\mathcal J$ that is not completely contained in $\mathcal J$, a contradiction. Therefore $i\notin X$, and so $A\cap X = A\setminus\{i\}$. Since $[t]\subseteq A$, we have that $X\cap [t] = [t]\setminus\{i\}$ and $X\cap [t-1]=[t-1]\setminus\{i\}$. Since we assumed that no such set exists in $\mathcal J$, we must have $X\in\mathcal V(\mathcal I,\mathcal J)$. But now this implies that $X\in\mathcal G_j$ for all $j\in[t]\setminus\{i\}$, and $X\notin\mathcal G_i$, which is a contradiction as all these sets $\cG_1,\dots,\cG_t$ are equal. Therefore, we must have $X\subseteq D$.

Since $D$ was chosen to be an arbitrary element of $S_A$, this means that for every $D\in S_A$, there exists $X\in\mathcal V(\mathcal I,\mathcal J)\cup\mathcal J$ such that $X\subseteq D$. Since any such $D$ has size $m$, we must have $X\in\mathcal V(\mathcal I,\mathcal J)$. Moreover, $X$ is the top of a diamond where the other 3 elements are in $\mathcal I$. Suppose that the two incomparable elements of the diamond are $R$ and $T$. We then must have $X=R\cup T$ by minimality. If both $R$ and $T$ would be contained in $A\setminus\{i\}$, then $X\subseteq A\setminus\{i\}\subsetneq A$, contradicting the $C_2$-free property of $\mathcal V(\mathcal I,\mathcal J)$. 

Therefore there exists a minimum integer $q$ and $Z_1,\dots,Z_q\in\mathcal I$ such that for every $D\in S_A$ there exists $a\in[q]$ such that $Z_a\subseteq D$, and $Z_a\setminus(A\setminus\{i\})\neq\emptyset$ (for the $R$ and $T$ above, the one not contained in $A\setminus\{i\}$ is one of the $Z$'s). Let $Y_a= Z_a\setminus (A\setminus\{i\})$. Therefore, for every $D\in S_A$, there exists $a\in[q]$ such that $Y_a\subseteq D\setminus (A\setminus\{i\})$. Consequently we have that for every $B\in \binom{[n]\setminus A}{m-|A|+1}$ there exists $a\in[q]$ such that $Y_a\subseteq B$. Note that this implies that $Y_a\subseteq [n]\setminus A$. 

Suppose that $q\leq n-m -1$. For every $a\in[q]$ we pick an element $j_a\in Y_a$, and consider the set $[n]\setminus(A\cup \{j_a:a\in [q]\})$. Note that $|[n]\setminus (A\cup \{j_a:a\in [q] \})|\geq n- |A|-q\geq m-|A|+1$, thus we can pick a set $B\subseteq [n]\setminus(A\cup \{j_a:a\in [q]\})$ of size $m-|A|+1$. But for such a $B$ we have, by construction, that $B\in\binom{[n\setminus A]}{m-|A|+1}$ and $B\not\supseteq Y_a$ for all $a\in [q]$, a contradiction. We therefore must have $q\geq n-m$.

But now this means that $|\mathcal I|\geq n-m$, which contradicts our assumption that $|\mathcal I\cup\mathcal J|\leq n-2m-1$. Therefore, for every $i\in [t-1]$ there exists an element $H\in \mathcal J$ such that $H\cap [t-1] = [t-1]\setminus \{i\}$, which completes the proof of the claim. 
\end{proof}
Putting everything together we have that $|\mathcal I\cup \mathcal J| = |\mathcal I|+|\mathcal J|\geq |\widehat{\mathcal I}|+1+|\widehat{\mathcal J}|+t-1=|\widehat{\mathcal I}|+|\widehat{\mathcal J}|+t$. If $n\geq 2m+t+1$ we have that $|\widehat{\mathcal I}|+|\widehat{\mathcal J}|=|\widehat{\mathcal I}\cup\widehat{\mathcal J}|\geq f(n-t,m)\geq n -t - 2m $ by our induction hypothesis, which implies that $|\mathcal I\cup \mathcal J|\geq n - 2m$. If on the other hand $n \leq 2m+t$ then $|\mathcal I\cup \mathcal J|\geq |\widehat{\mathcal I}|+|\widehat{\mathcal J}|+t\geq t\geq n-2m $. Therefore, in both cases we have $|\mathcal I\cup\mathcal J|\geq n-2m$, contradicting the initial assumption. Thus $f(n,m)\geq n-2m$, which completes the inductive step.
\end{proof}

\section{Proof of the main theorem}
The results in Section 2 and 3 are now aligned to obtain the desired lower bound for a diamond-saturated family. We will lower bound $G(\mathcal A)\cup\mathcal B$ with the help of Lemma~\ref{boundingsizecalAcalB}, after removing the set $W$, which we know how to upper bound by Lemma~\ref{niceproperty3}.

For clarity, we recall that $\mathcal B$ is the set of maximal elements of $\mathcal F$, $\mathcal A$ is the set of  minimal elements with the property that they are above a copy of $\mathcal V$ in $\mathcal F$, and do not contain any set of $\mathcal B$, and $G(\mathcal A)$ is comprised of all the sets in $\mathcal F$ that generate $\mathcal A$.
\begin{theorem}
Let $n\geq 1$, and $\mathcal F$ a diamond-saturated family with ground set $[n]$. Then $|\mathcal F|\geq\frac{n+1}{5}$.
\label{maintheorem}
\end{theorem}
\begin{proof}
If $\emptyset\in\mathcal F$, or $[n]\in\mathcal F$, then we know that $|\mathcal F|\geq n+1$, and so we may assume that neither $\emptyset$ or $[n]$ are in $\mathcal F$. This means that we are now in entirely in the regime of Section 2. Since $\frac{n+1}{5}<1$ for $n<4$, the result is trivially true for $n<4$, so we may assume that $n\geq 4$. Towards a contradiction, suppose that $\mathcal F <\frac{n+1}{5 }\leq \frac{n}{4}$, as $n\geq4$.

Recall that $W=\{i\in[n]:i\notin X\text{ for all } X\in\mathcal A\}$. By the second part of Lemma~\ref{niceproperty3}, we have that $|W|\leq 2|\mathcal F|-1\leq\frac{n}{2}-1$. Since $\mathcal A$ is a subset of the set of minimal elements that are above a copy of $\mathcal V$ in $\mathcal F$, if $A\in\mathcal A$, then $A=P\cup Q$ for some $P,Q\in G(\mathcal A)$. That means that if $i\in W$, then $i\notin X$ for all $X\in G(\mathcal A)$. Conversely, if $i\notin X$ for all $X\in G(\mathcal A)$, then by the same minimality argument we have $i\in W$. Therefore $W=\{i\in[n]:i\notin X\text{ for all }X\in G(\mathcal A)\}$. Let $\widehat{B}=\{B\setminus W:B\in\mathcal B\}$.

Moreover, by Lemma~\ref{niceproperty2}, we have that every set in $\mathcal A$ has size at most $|\mathcal F|$, and so every set in $G(\mathcal A)$ has size at most $|\mathcal F|$. Also, since every set in $\mathcal B$ has size at least $n-|\mathcal F|$, and $|\mathcal F|\leq\frac{n}{4}$, we must have that $G(\mathcal A)$ and $\mathcal B$ are disjoint. Furthermore, if $\widehat{B}\in\widehat{\mathcal B}$, then $|\widehat{B}|\geq n-|\mathcal F|-|W|=|[n]\setminus W|-|\mathcal F|$. This means that every set in $\widehat{B}$ has size at least $n+1-3|\mathcal F|$, which is greater than $|\mathcal F|$, as $|\mathcal F|\leq\frac{n}{4}$. Thus $\mathcal A$ and $\widehat{\mathcal B}$ are disjoint, as every element of $\mathcal A$ has size at most $|\mathcal F|\leq\frac{n}{4}$, and every element of $\widehat{\mathcal B}$ has size at least $|[n]\setminus W|-|\mathcal F|\geq\frac{n}{4}+1$.
\begin{claim5}$(G(\mathcal A),\widehat{\mathcal B})\in\mathcal L^*([n]\setminus W, |\mathcal F|)$.   
\end{claim5}
\begin{proof}
First, since $\mathcal L^*(X,m)$ is defined only when $2m+1\leq |X|$, we need to have $2|\mathcal F|+1\leq n-|W|$, or equivalently, $2|\mathcal F|+|W|+1\leq n$. This is true since $|W|\leq2|\mathcal F|-1$ and $|\mathcal F|\leq\frac{n}{4}$. Also, by construction, $G(\mathcal A)$ and $\widehat{\mathcal B}$ have ground set $[n]\setminus W$.

Since we already established that all sets of $G(\mathcal A)$ have size at most $|\mathcal F|$, we are left to prove that $\mathcal V(G(\mathcal A,\widehat{\mathcal B}),\widehat{\mathcal B})\in\mathcal L([n]\setminus W)$. Observe that if $X\in\mathcal V(G(\mathcal A),\widehat{\mathcal B})$, then, by minimality, it is the union of two sets if $G(\mathcal A)$, thus it cannot contain a subset of $\mathcal B$, as this would imply it contains a subset of $\widehat{\mathcal B}$. Hence we have that $\mathcal V(G(\mathcal A),\widehat{\mathcal B})\subseteq \mathcal A$. Moreover, since $G(\mathcal A)$ generates $\mathcal A$ and, by cardinality, no element of $\mathcal B$ or $\widehat{\mathcal B}$ can be a subset of an element of $\mathcal A$, we conclude that $\mathcal V(G(\mathcal A),\widehat{\mathcal B})=\mathcal A$.

We therefore must show that $(\mathcal A,\widehat{\mathcal B})\in\mathcal L([n]\setminus W, |\mathcal F|)$. We have already showed that $\mathcal A$ and $\widehat{\mathcal B}$ are disjoint, that every set of $\mathcal A$ has size at most $|\mathcal F|$ and every set of $\widehat{\mathcal B}$ has size at least $|n\setminus W|-|\mathcal F|$, and that $\mathcal A$ is a cover of $[n]\setminus W$ by construction. Suppose now that $\mathcal A\cup\widehat{\mathcal B}$ contains a copy of $C_2$, say $X_1\subsetneq X_2$, such that at least one of them is in $\mathcal A$. Since $\mathcal A$ is an antichain, we must have one element in $\mathcal A$ and one in $\widehat{\mathcal B}$, and since all sets of $\mathcal A$ have sizes less than the size of any element in $\widehat{\mathcal B}$, we must have $X_1\in\mathcal A$ and $X_2\in\widehat{\mathcal B}$, so $X_2=B\setminus W$ for some $B\in\mathcal B$. But now $X_1\subsetneq X_2=B\setminus W\subseteq B$, which contradicts the fact that $\mathcal A\cup\mathcal B$ is $C_2$-free (Lemma~\ref{niceproperty1}). Finally, suppose $X\in\mathcal P([n]\setminus W)$ is a set such that $|\mathcal F|\leq |X|\leq n-|W|-|\mathcal F|$. If $X\in\mathcal A$, we have $X\subseteq X$. If $X\in\mathcal B$, we have $X\setminus W\subseteq X$ and $X\setminus W\in\widehat{\mathcal B}$. Thus, we may assume that $X\notin \mathcal A\cup\mathcal B$, which, by Lemma ~\ref{niceproperty1}, implies that there exists $Y\in\mathcal A\cup\mathcal B$ such that $X\subset Y$, or $Y\subset X$. If $X\subset Y$, by cardinality we must have $Y\in\mathcal B$, and so $X\subset Y\setminus W$, and $Y\setminus W\in\widehat{\mathcal B}$. If on the other hand $Y\subset X$, we similarly must have $Y\in\mathcal A$, which finishes the proof of the claim.

\end{proof}
We now see that, since $G(\mathcal A)$ and $\mathcal B$ are disjoint and subsets of $\mathcal F$, $|\mathcal F|\geq |G(\mathcal A)|+|{\mathcal B}|\geq |G({\mathcal A})|+|\widehat{\mathcal B}|\geq f(n-|W|,|\mathcal F|)\geq n-|W|-2|\mathcal F|$, where the last inequality comes from Lemma~\ref{boundingsizecalAcalB}. Together with the fact that $|W|\leq 2|\mathcal F|-1$, we get that $|\mathcal F|\geq n+1-4|\mathcal F|$, or equivalently $|\mathcal F|\geq\frac{n+1}{5}$, a contradiction. Therefore, $\text{sat}^*(n,\mathcal D_2)\geq\frac{n+1}{5}$, ad desired.
\end{proof}
Finally, Theorem~\ref{maintheorem} together with the fact that $\text{sat}^*(n,\mathcal D_2)\leq n+1$, gives the following.
\begin{theorem} For all $n\geq1$, we have that $\frac{n+1}{5}\leq\text{sat}^*(n,\mathcal D_2)\leq n+1$.
\end{theorem}
\section{Concluding remarks}
We would like to mention that the linearity of the diamond, combined with Proposition 5 in \cite{gluing} shows that a large class of posets have saturation number at least linear. Recall that, given two finite posets $\mathcal P$ and $\mathcal Q$, we form the poset $\mathcal P*\mathcal Q$ by putting a copy of $\mathcal P$ entirely on top of a copy of $\mathcal Q$. This is also know as the linear sum of $\mathcal P$ and $\mathcal Q$. Let $\bullet$ represent the one element poset.
\begin{proposition}[Proposition 5 in \cite{gluing}]
Let $\mathcal P_1$ and $\mathcal P_2$ be any non-empty posets such that $\mathcal P_1$ does not have a unique maximal element and $\mathcal P_2$ does not have a unique minimal element. Then $\text{sat}^*(n,\mathcal P_2*\mathcal P_1)\geq\max\{\text{sat}^*(n,\mathcal P_2*\bullet), \text{sat}^*(n,\bullet*\mathcal P_1)\}$.
\end{proposition}
After two straightforward applications of the above proposition, we get the following.
\begin{corollary}
Let $\mathcal P_1$ and $\mathcal P_2$ be any non-empty posets such that $\mathcal P_1$ does not have a unique maximal element and $\mathcal P_2$ does not have a unique minimal element. Then $\text{sat}^*(n,\mathcal P_2*\mathcal A_2*\mathcal P_1)\geq\text{sat}^*(n,\mathcal D_2)\geq\frac{n+1}{5}$, where $\mathcal A_2$ is the antichain of size 2.
\end{corollary}

Finally, since the saturation number for a multipartite complete poset that does not have two consecutive layers of size 1 was shown to be $O(n)$ \cite{gluing}, the above corollary implies that the saturation number for all such posets, that also have a layer of size 2, is linear.
\begin{corollary}
Let $k\geq2$ and $n_1, n_2,\dots,n_k$ be positive integers. Let $K_{n_1,\dots,n_k}$ denote the complete poset with $k$ layers of sizes $n_1,\dots,n_k$, in this order, where $n_1$ is the size of the bottom layer. If there exists no $i\in[k]$ such that $n_i=n_{i+1}=1$, and there exists $j\in[k]$ such that $n_j=2$, then $\text{sat}^*(n,K_{n_1,\dots,n_k})=\Theta(n)$.
\end{corollary}
\bibliographystyle{amsplain}
\bibliography{references}
\Addresses
\end{document}